\UseRawInputEncoding
\documentclass[12pt]{article}
\usepackage{bbm}
\usepackage{latexsym}
\usepackage{mathrsfs}
\usepackage{graphicx}
\usepackage{epstopdf}
\usepackage{subfigure}
\usepackage{amssymb}
\usepackage{amsmath}
\usepackage{amsthm}
\usepackage{color}
\usepackage{cite}
\usepackage{float}
\usepackage{indentfirst}
\usepackage{caption}
\usepackage{float}
\usepackage{pdflscape}
\usepackage{geometry}
\usepackage{enumerate}

\newcommand{\FS}[2]{\displaystyle\frac{#1}{#2}}

\newtheoremstyle{lemma}{\topsep}{\topsep}%
     {}
     {}
     {\bfseries}
     {}
     {0.1em}
     {\thmname{#1}\thmnumber{ #2}\thmnote{ #3}}
\theoremstyle{lemma}  

\newtheorem{theorem}{Theorem}[section]    
\newtheorem{lemma}[theorem]{Lemma}
\newtheorem{corollary}[theorem]{Corollary}

\numberwithin{equation}{section}

\usepackage{latexsym, bm}
\title
{\textbf{Forcing and anti-forcing polynomials of a polyomino graph}
\thanks{Supported by the Natural Science Foundation of Ningxia, China (Grant No. 2019AAC03124), and the Science and Technology Research Foundation of the Higher Education Institutions of Ningxia, China (Grant No. NGY2018-139).}}
\author{{Kai Deng$^{a,}$\thanks{
Corresponding author.}
, Huazhong L\"{u}$^{b}$, Tingzeng Wu$^{c}$}\\
{\footnotesize $^{a}$ School of Mathematics and Information Science,
North Minzu University,}
\\{\footnotesize Yinchuan, Ningxia 750021, P. R.~China}
\\{\footnotesize E-mail: dengkai04@126.com}
\\{\footnotesize $^{b}$ School of Mathematical Sciences,
University of Electronic Science and Technology of China,}
\\{\footnotesize Chengdu, Sichuan 610054, P.R. ~China}
\\{\footnotesize E-mail:  lvhz08@lzu.edu.cn}
\\{\footnotesize $^{c}$ School of Mathematics and Statistics,
Qinghai Nationalities University,}
\\{\footnotesize Xining, Qinghai 810007, P.R.~China
}
\\{\footnotesize E-mail: mathtzwu@163.com}
}
\date{}

\begin{document}
\maketitle
\begin{abstract}
The forcing number of a perfect matching $M$
in a graph $G$ is the smallest number of edges inside $M$
that can not be contained in other perfect matchings.
The anti-forcing number of $M$ is the smallest number
of edges outside $M$ whose removal results in a
subgraph with a single perfect matching,
that is $M$.
Recently, in order to investigate the distributions of forcing numbers and anti-forcing numbers,
the forcing polynomial and anti-forcing polynomial were proposed, respectively.
In this work, the forcing and anti-forcing polynomials
of a polyomino graph are obtained.
As consequences, the forcing and anti-forcing spectra of this polyomino graph are determined,
and the asymptotic behaviors on the degree of freedom
and the sum of all anti-forcing numbers are revealed, respectively.

\vskip 0.1 in

\end{abstract}

\section{Introduction}

A {\em perfect matching} in a graph is a set of independent edges
that saturates each vertex.
The perfect matching arose in many fields,
such as a dimer covering in statistical physics \cite{Kasteleyn},
a Kekul\'{e} structure in quantum chemistry \cite{Cyvin},
and a solution for optimal assignment \cite{LP}, etc.
In 1985--1987 Klein and Randi\'{c} \cite{Klein, Randic}
observed that a few number of fixed double bonds of a Kekul\'{e} structure
can determine the whole Kekul\'{e} structure,
and defined the least number of fixed double bonds as
the \emph{innate degree of freedom} of this Kekul\'{e} structure.
This concept was extended to a perfect matching by Harary et al. \cite{Harary}
in 1991,
and renamed as \emph{forcing number}.
From the opposite point of view,
Vuki\v{c}evi\'{c} and Trinajsti\'{e} \cite{VT1, VT2}
proposed the anti-forcing number of a graph,
that was extended to a single perfect matching in 2015 \cite{Zh2}.
Generally, computing the forcing and anti-forcing numbers of a perfect matching
are both NP-complete\cite{Adams, ZhangD2}.
Recently, the forcing polynomial \cite{ZhangZhao} and anti-forcing polynomial \cite{Hwang}
were proposed to investigate the distributions of forcing and anti-forcing
numbers of all perfect matchings, respectively.
In this paper,
we will calculate the forcing polynomial
and anti-forcing polynomial of a polyomino graph, respectively.

Suppose graph $G$ has a perfect matching $M$.
A subset $S \subseteq M$ is called a {\em forcing set} of $M$
if other perfect matchings in $G$ do not contain $S$.
That is, $M$ is determined by an edge set inside of $M$.
Thus the forcing number of $M$ is the size of a minimum forcing set,
denoted by $f(G,M)$.
The \emph{forcing spectrum} \cite{Adams, Afshani} of $G$ is
the collection of forcing numbers of all perfect matchings,
denoted by Spec$_f(G)$.
The minimum and maximum values in Spec$_f(G)$ are called
the \emph{minimum} and \emph{maximum forcing numbers} of $G$,
denoted by $f(G)$ and $F(G)$, respectively.
Finding the minimum forcing number was proved to be a NP-complete problem \cite{Afshani},
but the special structure of graph helps
to compute it.
The minimum forcing numbers of some special graphs were obtained,
such as $2n\times2n$ grids \cite{Pachter},
hypercubes \cite{Matthew, Diwan},
$2m\times2n$ toric grids \cite{Matthew},
toric hexagonal systems \cite{Wang}, etc.
On the other hand,
the maximum forcing number was considered,
such as cylindrical grids \cite{Afshani, Jiang1},
$2m\times 2n$ toric grids \cite{Kleinerman},
4-8 lattices \cite{Jiang1}, etc.
For a hexagonal system,
Xu et al. \cite{ZhangF} confirmed that the maximum forcing number is equal to the Clar number,
that is true for polyomino graphs \cite{Zhou2} and (4,6)-fullerenes \cite{Shi1} as well.
The forcing spectra of some graphs were discussed,
such as hypercubes \cite{Adams}, grid graphs \cite{Afshani},
tubular boron-nitrogen fullerene graphs \cite{Jiang}, hexagonal systems \cite{ZhangD}, etc.
Especially, Randi\'{c}, Vuki\v{c}evi\'{c} and Gutman \cite{c70, c72, c60}
calculated the forcing spectra of fullerenes
C$_{60}$, C$_{70}$ and C$_{72}$,
showing that the Kekul\'{e} structure contributing more to the stability of molecule
often has larger innate degree of freedom.
There are few results on forcing polynomial,
Zhao and Zhang \cite{Zhaos, ZhangZhao, Zhaos2} considered
some special hexagonal systems, and $2\times n$ and $3\times2n$  grids.
For more results, see
\cite{Li, Che, Lam, Che1, Shi, Zhou1, Zhou, Ye}.

Let $E(G)$ denote the edge set of a graph $G$,
and let $M$ be a perfect matching of $G$.
For a subset $S^{\prime}\subseteq E(G)\setminus M$,
if $G-S^{\prime}$ (the graph obtained by removing all the edges of $S^\prime$ from $G$)
has a single perfect matching, that is $M$,
then $S^\prime$ is named an \emph{anti-forcing set} of $M$.
In other words, $M$ is fixed by an edge set outside of $M$.
The size of a minimum anti-forcing set
is named  the \emph{anti-forcing number} of $M$, denoted by $af(G, M)$.
The collection of anti-forcing numbers of all perfect matchings is
called the \emph{anti-forcing spectrum} of $G$,
denoted by Spec$_{af}(G)$.
The minimum and maximum values in Spec$_{af}(G)$
are called \emph{minimum} and \emph{maximum anti-forcing numbers} of $G$,
denoted by $af(G)$ and $Af(G)$, respectively.
Actually, as early in 1997 Li \cite{Li2} had
described the extremal hexagonal systems of minimal anti-forcing number 1.
The minimum anti-forcing numbers of some chemical graphs were discussed,
such as hexagonal systems \cite{Deng1, Deng2, VT1, VT2},
catacondensed phenylene \cite{ZhangQQ}, fullerene graphs \cite{Yang}, etc.
Recently, Lei et al. \cite{Zh2} proved that
the maximum anti-forcing number is equal to Fries number for a hexagonal system,
that is true for (4,6)-fullerenes \cite{Shi1} as well.
Furthermore, Deng and Zhang \cite{ZhangD1} showed that
the cyclomatic number is an attainable upper bound on the maximum anti-forcing number,
and characterized the extremal graphs.
Afterwards, Shi and Zhang \cite{Shi2} gave another achievable upper bound,
and got the maximum anti-forcing numbers of hypercube
and its two variants.
The anti-forcing spectra of some hexagonal systems
were confirmed to be continuous, such as monotonic constructable hexagonal systems \cite{ZhangD2}, catacondensed hexagonal systems \cite{ZhangD3}, etc.
For anti-forcing polynomial,
a few number of graphs are considered, such as extremal hexagonal systems with forcing number 1,
and $2\times n$ and $3\times2n$ grids.

In section 2, as preparations,
some basic results on forcing and anti-forcing
numbers are introduced,
and some useful properties of the polyomino graph $G_n$ (see Fig. \ref{SQ1}) are discussed.
In section 3, we get the
forcing polynomial of $G_n$,
as corollaries, Spec$_{f}(G_n)=[n,2n]$,
and the asymptotic behavior of degree of freedom
of $G_n$ is revealed.
In section 4, we obtain the anti-forcing polynomial of $G_n$,
as consequences, Spec$_{af}(G_n)=[n,3n]$,
and the asymptotic behavior of the sum
of all anti-forcing numbers is showed.

\section{Preliminaries}

\begin{figure}[http]
\centering
\subfigure[$G_n$]{
    \label{SQ1}
 \includegraphics[height=40mm]{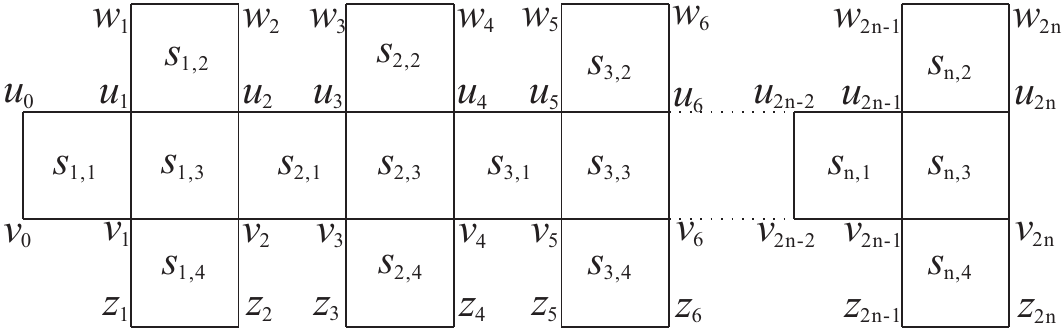}}
\hspace{0mm}
\subfigure[$H_n=G_n\ominus\{v_0,u_0\}$]{
    \label{SQ2}
~~~~~~~~~\includegraphics[height=40mm]{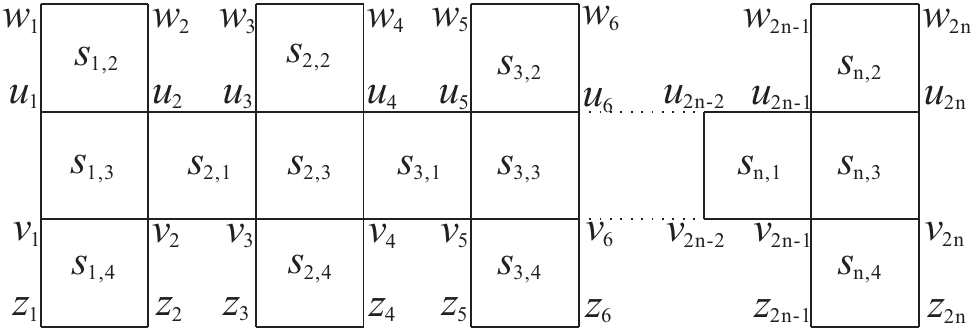}}
 \caption{The polyomino graphs $G_n$ and $H_n$.}
\label{SQ}
\end{figure}

A {\em polyomino graph} is a 2-edge-connected subgraph of the infinite plane grid
such that the periphery of every interior face is a square.
The polyomino graph is an important plane bipartite graphs,
which is studied in many combinatorial problems
\cite{Berge, Cockayne, Motoyama, Kasteleyn, Zhang0}.
The polyomino graph $G_n$ as shown in Fig. \ref{SQ1}
consisting of $4n$ squares,
that is a subgraph of $4\times(2n+1)$ grids,
and the vertices of $G_n$ are labeled by
$u_0, v_0, u_i, v_i, w_i, z_i$, $i=1,2,\cdots,2n$.
The polyomino graph $H_n$ is obtained by removing
the leftmost square $s_{1,1}$ from $G_n$ (see Fig. \ref{SQ2}),
that is a subgraph of $4\times 2n$ grids.

The number of perfect matchings of a graph $G$ is denoted by $\Phi(G)$.
Let $G_0$ denote the null graph.
Then $\Phi(G_0)=1$.
For a perfect matching $M$,
a cycle $C$ is called an {\em $M$-alternating cycle}
if the edges of $C$ appear alternately in $M$ and $E(C)\setminus M$.
Let $W$ be a set of vertices,
and let $G\ominus W$ denote the subgraph generated by
deleting all the vertices of $W$ and their incident edges from $G$.

\begin{lemma}\label{pm}{\bf .}
For $n\geq2$,
\begin{equation}\label{pms}
\Phi(G_n)=6\Phi(G_{n-1})-4\Phi(G_{n-2}),
\end{equation}
where $\Phi(G_{0})=1$ and $\Phi(G_{1})=6$.
\end{lemma}

\begin{proof}
Let $M^\prime$ be a perfect matching of $H_n$.
If $u_1v_1\in M^\prime$,
then $w_1w_2, z_1z_2\in M^\prime$.
Thus the restriction of $M^\prime$ on $G_{n-1}=H_n\ominus\{u_1,v_1,w_1,w_2,z_1,z_2\}$
is a perfect matching of $G_{n-1}$.
If $u_1v_1\not\in M^\prime$,
then $s_{1,2}$ and $s_{1,4}$ both
are $M^\prime$-alternating squares,
and the restriction of $M^\prime$ on $H_{n-1}=H_{n}\ominus\{w_1,w_2,u_1,u_2,v_1,v_2,z_1,z_2\}$
is a perfect matching of $H_{n-1}$ (see Fig. \ref{SQ2}).
Note that there are four different cases such that $s_{1,2}$ and $s_{1,4}$ both
are $M^\prime$-alternating. Therefore
\begin{equation}\label{r1}
\Phi(H_n)=\Phi(G_{n-1})+4\Phi(H_{n-1}).
\end{equation}

Let $M$ be a perfect matching of $G_n$.
If $u_0v_0\in M$,
then the restriction of $M$ on $H_{n}=G_n\ominus\{u_0,v_0\}$
is a perfect matching of $H_{n}$ (see Fig. \ref{SQ}).
If $u_0v_0\not\in M$,
then $u_0u_1,v_0v_1,w_1w_2,z_1z_2\in M$,
and the restriction of $M$ on $G_{n-1}= G_n\ominus\{u_0,u_1,v_0,v_1,w_1,w_2,$
$z_1,z_2\}$
is a perfect matching of $G_{n-1}$.
So $\Phi(G_n)=\Phi(H_n)+\Phi(G_{n-1})$, we have
$\Phi(H_n)=\Phi(G_n)-\Phi(G_{n-1})$ and $\Phi(H_{n-1})=\Phi(G_{n-1})-\Phi(G_{n-2})$,
substituting them into Eq. (\ref{r1}), then Eq. (\ref{pms}) is obtained.
\end{proof}

\begin{theorem}\label{PM}{\bf .}
\begin{equation}\label{PMF}
\Phi(G_n)=\FS{5-3\sqrt{5}}{10}(3-\sqrt{5})^n+\FS{5+3\sqrt{5}}{10}(3+\sqrt{5})^n.
\end{equation}
\end{theorem}

\begin{proof}
According to Lemma \ref{pm},
the characteristics equation of recurrence formula (\ref{pms})
is $x^2-6x+4=0$,
with roots $3\pm \sqrt{5}$.
Suppose the general solution of formula (\ref{pms}) is
$\Phi(G_n)=\lambda_0(3-\sqrt{5})^n+\lambda_1(3+\sqrt{5})^n$,
by the initial conditions $\Phi(G_2)=32$ and $\Phi(G_3)=168$,
we can obtain $\lambda_0=\frac{5-3\sqrt{5}}{10}$ and $\lambda_1=\frac{5+3\sqrt{5}}{10}$,
so Eq. (\ref{PMF}) holds for $n\geq2$.
We can check that the equation also holds for $n=0,1$,
so the proof is completed.
\end{proof}

Let $M$ be a perfect matching in graph $G$,
and $c(M)$ denote the largest number of disjoint $M$-alternating cycles.
Pachter and Kim \cite{Pachter} showed the following result by means of
the minimax theorem on feedback set \cite{Younger}.

\begin{theorem}\label{cycle}{\rm\cite{Pachter}}{\bf .}
Let $M$ be a perfect matching of a plane bipartite graph $G$. Then $f(G,M)=c(M)$.
\end{theorem}

An \emph{$M$-resonant set} with respect to a perfect matching $M$ of a polyomino graph $G$
is a collection of independent $M$-alternating squares.
Let $s(M)$ denote the size of a maximal $M$-resonant set.
Then the maximum $s(M)$ over all perfect matchings is called the \emph{Clar number} of $G$, denoted by $cl(G)$.

\begin{theorem}\label{maxforcing}\cite{Zhou2}{\bf .}
Suppose $G$ is a polyomino graph that has a perfect matching. Then $F(G)=cl(G)$.
\end{theorem}

\begin{lemma}\label{resonance}\cite{Zhang}{\bf .}
Let $C$ be an $M$-alternating cycle with respect to
a perfect matching $M$ in a planar bipartite graph $G$.
Then there exists a face in the interior of $C$ whose periphery is an $M$-alternating cycle.
\end{lemma}

Suppose $M$ is a perfect matching of a polyomino graph,
and $\mathcal{A}$ is a collection of $M$-alternating cycles.
Let $I(\mathcal{A})=\sum_{C\in\mathcal{A}}I(C)$,
where $I(C)$ is the number of squares in the interior of cycle $C$.

\begin{lemma}\label{mrset}{\bf .}
Let $M$ be a perfect matching of $G_n$.
Then $f(G_n,M)=s(M)$.
\end{lemma}

\begin{proof}
Let $\mathcal{A}$ be a maximal set of disjoint $M$-alternating cycles
such that $I(\mathcal{A})$ is as small as possible.
Then $\mathcal{A}$ is an $M$-resonant set.
Otherwise, $\mathcal{A}$ must contain a non-square cycle $C$.
By Lemma \ref{resonance}, there is an $M$-alternating square
$s$ in the interior of $C$.
Since each vertex is on the periphery of $G_n$,
$\mathcal{A}^\prime =(\mathcal{A}\setminus\{C\})\cup\{s\}$ is
another maximal set of disjoint $M$-alternating cycles,
however,
$I(\mathcal{A}^\prime)<I(\mathcal{A})$, a contradiction.
Therefore $|\mathcal{A}|\leq s(M)$.
On the other hand,
by Theorem \ref{cycle}, $f(G_n,M)=|\mathcal{A}|\geq s(M)$.
So  $f(G_n,M)=s(M)$.
\end{proof}

Suppose that $M$ is a perfect matching in a graph $G$.
A \emph{compatible $M$-alternating set} $\mathcal{A}^\prime$ is
a collection of $M$-alternating cycles
such that any two cycles of $\mathcal{A}^\prime$
are disjoint, or overlap only on edges in $M$.
Let $c^{\prime}(M)$ denote the size of a maximal compatible $M$-alternating set.

\begin{theorem}\label{minimax2}\cite{Zh2}{\bf .}
Suppose that $M$ is a perfect matching in a planar bipartite graph $G$.
Then $af(G,M)=c^{\prime}(M)$.
\end{theorem}

For a planar bipartite graph,
two cycles $C_1$ and $C_2$
in a compatible $M$-alternating set $\mathcal{A}^\prime$ are called \emph{crossing}
if $C_1$ enters the interior of $C_2$ from the exterior of $C_2$.
If any two cycles of $\mathcal{A}^\prime$ are no crossing,
then $\mathcal{A}^\prime$ is called
\emph{non-crossing}.

\begin{lemma}\label{minimax9}\cite{Zh2,ZhangD1}{\bf .}
Suppose $G$ is a planar bipartite graph that has a perfect matching $M$.
Then there is a non-crossing compatible $M$-alternating set $\mathcal{A}^\prime$
such that $af(G,M)=|\mathcal{A}^\prime|$.
\end{lemma}

\begin{figure}[http]
\centering
\subfigure[$L_k$]{
    \label{Lk}
 \includegraphics[height=10.5mm]{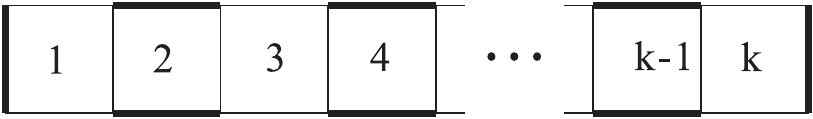}}
\hspace{0mm}
\subfigure[$W_1$, $W_2$ and $W_r$]{
    \label{Hr}
~~~~~~~~~\includegraphics[height=28mm]{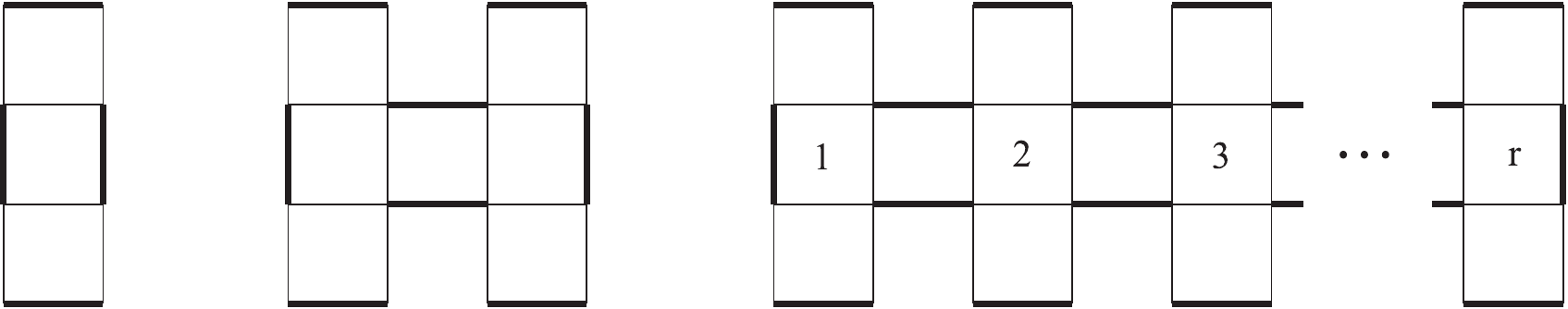}}
 \caption{The illustrations of Lemma \ref{minir}, where bold edges belong to a perfect matching.}
\label{LH}
\end{figure}

The anti-forcing number of a perfect matching $M$ in $G_{n}$
is related with
the number of substructures $L_k$ and $W_r$
as shown in Fig. \ref{LH}.
$L_k$ is a straight chain with $k$ squares,
$1\leq k\leq 2n-1$ and $k$ is odd.
The restriction of $M$ on $L_k$ is a perfect matching of $L_k$,
and the periphery of $L_k$  is
an $M$-alternating cycle which is compatible
with $\frac{k-1}{2}$ $M$-alternating squares in its interior (see Fig. \ref{Lk}).
$W_r$ is isomorphic to $H_r$,
the restriction of $M$ on $W_r$ is a perfect matching of $W_r$,
and $W_r$ contains the substructure $L_{2r-1}$
such that the peripheries of $W_r$ and $L_{2r-1}$
are non-crossing compatible $M$-alternating cycles,
and both of them are compatible with $r-1$ $M$-alternating squares in their interiors
(see Fig. \ref{Hr}).

\begin{lemma}\label{minir}{\bf .}
Let $M$ be a perfect matching in $G_n$, and $\mathcal{A}^\prime$ be
a maximal non-crossing compatible $M$-alternating set
such that $I(\mathcal{A}^\prime)$ is as small as possible.
Then each $M$-alternating square must belong to $\mathcal{A}^\prime$,
and any non-square cycle of $\mathcal{A}^\prime$
either is the periphery of an $L_k$ or the periphery of a $W_r$.
\end{lemma}

\begin{proof}
Suppose $s$ is an $M$-alternating square that is not in $\mathcal{A}^\prime$.
Then $s$ is not compatible with a cycle $C\in \mathcal{A}^\prime$.
Since any two $M$-alternating squares are compatible with each other,
$C$ is a non-square cycle.
Therefore $\mathcal{A}^{\prime\prime}=(\mathcal{A}^\prime\setminus\{C\})\cup\{s\}$
is another maximal non-crossing compatible $M$-alternating set,
but $I(\mathcal{A}^{\prime\prime})<I(\mathcal{A}^{\prime})$,
a contradiction. So each $M$-alternating square must belong to $\mathcal{A}^\prime$.

Let $C\in\mathcal{A}^\prime$ be a non-square cycle,
$M(C)=M\cap E(C)$.
Note that none of these vertical edges $w_iu_i$ and $v_jz_j$ for $1\leq i,j\leq 2n$
belong to $M(C)$,
otherwise $s_{\lceil\frac{i}{2}\rceil,2}$ or
$s_{\lceil\frac{j}{2}\rceil,4}$ will be an $M$-alternating square
that is not compatible with $C$ (see Fig. \ref{SQ1}).
Since $C$ is $M$-alternating,
$M(C)$ contains just two vertical edges $u_iv_i$ and $u_jv_j(0\leq i<j\leq 2n)$
and $i+j$ is odd.
If $i$ is even, then $C=u_iu_{i+1}\cdots u_jv_jv_{j-1}\cdots v_iu_i$ is the periphery of an $L_{j-i}$.
If $i$ is odd, then $C$ is the periphery of an $L_{j-i}$ or
$C=u_iw_i$
$w_{i+1}u_{i+1}u_{i+2}w_{i+2}w_{i+3}u_{i+3}\cdots
u_{j-1}w_{j-1}w_{j}u_{j}v_{j}z_{j}z_{j-1}v_{j-1}v_{j-2}z_{j-2}z_{j-3}v_{j-3}\cdots
v_{i+1}z_{i+1}z_{i}v_{i}u_{i}$
is the periphery of a $W_{\lceil\frac{j-i}{2}\rceil}$.
\end{proof}

\noindent\textbf{Remark 1.} The same as above, we can show that
Lemmas \ref{mrset} and \ref{minir} also hold for $H_n$.

\section{Forcing polynomials}

 The \emph{forcing polynomial} \cite{ZhangZhao} of a graph
$G$ is defined as follows :

\begin{equation}\label{eq1}
F(G,x)=\sum_{M\in \mathcal{M}(G)} x^{f(G,M)},
\end{equation}
where $\mathcal{M}(G)$ is the collection of all perfect matchings of $G$.

By Eq. $(\ref{eq1})$, another expression
is immediately obtained:

\begin{equation}\label{eq2}
\nonumber F(G,x)=\sum_{i=f(G)}^{F(G)} w(G,i)x^{i},
\end{equation}
where $w(G, i)$ is the number of perfect matchings with the forcing number $i$.

The forcing polynomial can count the number of perfect matchings with the same forcing number,
in other words, the distribution of all forcing numbers is revealed.
Moreover, $\Phi(G)=F(G,1)$,
and the \emph{degree of freedom} of $G$, $IDF(G)=\frac{d}{dx}F(G,x)|_{x=1}$,
i.e. the sum of forcing numbers of all perfect matchings,
which can estimate the resonance energy of a molecule \cite{Randic}.
Obviously, if $G$ is a null graph or a graph has a unique perfect matching,
then $F(G,x)=1$.

The following theorem give a recurrence relation for forcing polynomial of $G_n$.

\begin{theorem}\label{main1}{\bf .}
For $n\geq3$,
\begin{equation}\label{f1}
F(G_n,x)=(4x^2+3x)F(G_{n-1},x)-(8x^3+2x^2)F(G_{n-2},x)+4x^3F(G_{n-3},x),
\end{equation}
where $F(G_0,x)=1,F(G_1,x)=4x^2+2x,F(G_2,x)=16x^4+12x^3+4x^2$.
\end{theorem}

\begin{proof}
We divide $\mathcal{M}(G_n)$ in two subsets:
$\mathcal{M}_1=\{M\in\mathcal{M}(G_n)|u_0v_0\not\in M\}$ and $\mathcal{M}_2=\{M\in\mathcal{M}(G_n)|u_0v_0\in M\}$.
If $M\in \mathcal{M}_1$,
then $\{u_0u_1,v_0v_1,w_1w_2,z_1z_2\}\subseteq M$
and $s_{1,1}$ is an $M$-alternating square.
Note that the restriction $M^\prime$ of $M$ on $G_{n-1}=G_n\ominus\{u_0,u_1,v_0,v_1,w_1,w_2,z_1,z_2\}$
is a perfect matching of $G_{n-1}$.
Let $\mathcal{A}^\prime$ be a maximum $M^\prime$-resonance set
of $G_{n-1}$.
Then $\mathcal{A}^\prime\cup\{s_{1,1}\}$ is a maximum
$M$-resonance set of $G_{n}$ (see Fig. \ref{SQ1}).
By Lemma  \ref{mrset}, $f(G_n,M)=1+f(G_{n-1},M^\prime)$.
According to Eq. (\ref{eq1}),
\begin{eqnarray}\label{m1}
\nonumber\sum_{M\in \mathcal{M}_1}x^{f(G_n,M)}&=&\sum_{M^{\prime}\in \mathcal{M}(G_{n-1})}x^{1+f(G_{n-1},M^\prime)}\\
&=&x\sum_{M^{\prime}\in \mathcal{M}(G_{n-1})}x^{f(G_{n-1},M^\prime)}
=xF(G_{n-1},x).
\end{eqnarray}

Now suppose $M\in \mathcal{M}_2$,
we divide $\mathcal{M}_2$ in two subsets:
$\mathcal{M}_{2,1}=\{M\in \mathcal{M}_2|u_1v_1\in M\}$
and $\mathcal{M}_{2,2}=\{M\in \mathcal{M}_2|u_1v_1\not\in M\}$.
If $M\in\mathcal{M}_{2,1}$,
then the restriction $M^{\prime\prime}$ of $M$
on $G_{n-1}=G_n\ominus\{u_0,u_1,v_0,v_1,w_1,w_2,z_1,z_2\}$ is
a perfect matching of $G_{n-1}$.
Similarly,

\begin{eqnarray}\label{m2}
\sum_{M\in\mathcal{M}_{2,1}}x^{f(G_n,M)}=
\sum_{M^{\prime\prime}\in\mathcal{M}(G_{n-1})}x^{1+f(G_{n-1},M^{\prime\prime})}
=xF(G_{n-1},x).
\end{eqnarray}
If $M\in \mathcal{M}_{2,2}$,
then $\{s_{1,2},s_{1,4}\}$ is an $M$-resonance set
and the restriction $M^{\prime\prime\prime}$ of $M$
on $H_{n-1}=G_n\ominus\{u_0,u_1,v_0,v_1,w_1,w_2,z_1,z_2,u_2,v_2\}$
is a perfect matching of $H_{n-1}$ (see Fig. \ref{SQ}).
Let $\mathcal{A}^{\prime\prime\prime}$ be a maximum
$M^{\prime\prime\prime}$-resonance set of $H_{n-1}$.
Then $\mathcal{A}^{\prime\prime\prime}\cup\{s_{1,2},s_{1,4}\}$
is a maximum $M$-resonance set of $G_{n}$.
By Lemma \ref{mrset} and Remark 1, $af(G_n,M)=2+|\mathcal{A}^{\prime\prime\prime}|=2+f(H_{n-1},M^{\prime\prime\prime})$.
Note that there are four independent subcases
such that $\{s_{1,2},s_{1,4}\}$ is an $M$-resonance set,
so
\begin{eqnarray}\label{m3}
\nonumber\sum_{M\in\mathcal{M}_{2,2}}x^{af(G_n,M)}
&=&4\sum_{M^{\prime\prime\prime}\in\mathcal{M}(H_{n-1})}x^{2+f(H_{n-1},
M^{\prime\prime\prime})}\\
&=&4x^2F(H_{n-1},x).
\end{eqnarray}

According to Eqs. (\ref{m1}), (\ref{m2}) and (\ref{m3}),
\begin{eqnarray}\label{m4}
\nonumber F(G_n,x)&=&\sum_{M\in \mathcal{M}(G_{n})}x^{f(G_n,M)}\\
\nonumber&=&\sum_{M\in \mathcal{M}_1}x^{f(G_n,M)}+\sum_{M\in \mathcal{M}_2}x^{f(G_n,M)}\\
\nonumber&=&\sum_{M\in \mathcal{M}_1}x^{f(G_n,M)}+\sum_{M\in \mathcal{M}_{2,1}}x^{f(G_n,M)}
+\sum_{M\in \mathcal{M}_{2,2}}x^{f(G_n,M)}\\
&=&2xF(G_{n-1},x)+4x^2F(H_{n-1},x).
\end{eqnarray}

On the other hand,
we divide $\mathcal{M}(H_n)$ in two subsets:
$\mathcal{N}_1=\{M\in\mathcal{M}(H_n)|u_1v_1\in M\}$
and $\mathcal{N}_2=\{M\in\mathcal{M}(H_n)|u_1v_1\not\in M\}$.
Further, $\mathcal{N}_1$ can be divided into two subsets：
$\mathcal{N}_{1,1}=\{M\in\mathcal{N}_1|u_2v_2\in M \}$
and $\mathcal{N}_{1,2}=\{M\in\mathcal{N}_1|u_2v_2\not\in M\}$.
Similarly,
$\sum_{M\in\mathcal{N}_{1,1}}x^{f(H_n,M)}=xF(H_{n-1},x)$,
$\sum_{M\in\mathcal{N}_{1,2}}x^{f(H_n,M)}=xF(G_{n-2},x)$,
and $\sum_{M\in\mathcal{N}_{2}}x^{f(H_n,M)}=4x^2F(H_{n-1},x)$.
Therefore
\begin{eqnarray}\label{m5}
F(H_n,x)=(4x^2+x)F(H_{n-1},x)+xF(G_{n-2},x).
\end{eqnarray}

By Eq. (\ref{m5}) minus Eq. (\ref{m4}),
we can obtain
\begin{equation*}
F(H_n,x)-xF(H_{n-1},x)=F(G_n,x)-2xF(G_{n-1},x)+xF(G_{n-2},x),
\end{equation*}
thus
\begin{equation}\label{m6}
4xF(H_n,x)-4x^2F(H_{n-1},x)=4xF(G_n,x)-8x^2F(G_{n-1},x)+4x^2F(G_{n-2},x).
\end{equation}
According to  Eq. (\ref{m4}),
$4x^2F(H_{n-1},x)=F(G_n,x)-2xF(G_{n-1},x)$,
substituting it into  Eq. (\ref{m6}),
we have
\begin{equation*}
4xF(H_n,x)=(4x+1)F(G_n,x)-(8x^2+2x)F(G_{n-1},x)+4x^2F(G_{n-2},x),
\end{equation*}
so
\begin{equation*}
4xF(H_{n-1},x)=(4x+1)F(G_{n-1},x)-(8x^2+2x)F(G_{n-2},x)+4x^2F(G_{n-3},x),
\end{equation*}
substituting it into  Eq. (\ref{m4}),
then Eq. (\ref{f1}) is obtained.
\end{proof}

By Theorem \ref{main1}, we can deduce an explicit expression as follow.

\begin{theorem}\label{main2}{\bf .}
$F(G_n,x)=$
\begin{eqnarray}\label{m7}
\nonumber 2^{2n}x^{2n}&+&x^n\sum_{m=0}^{n-1}\Bigg{(}\sum_{i=max\{\lceil\frac{n}{3}\rceil,m\}}^{n}
\sum_{j=max\{\lceil\frac{n-i}{2}\rceil,m+n-2i\}}^{n-i}
\sum_{k=m-(i+2j-n)}^{m}(-1)^{i+2j-n}2^{n+2m-i}\\
\nonumber&~&\cdot3^{i-j-k}{i \choose j}{j\choose n-i-j}{i-j\choose k}{i+2j-n\choose m-k}-\\
&~&\sum_{i=max\{\lceil\frac{n-1}{3}\rceil,m\}}^{n-1}\sum_{j=max\{\lceil\frac{n-i-1}{2}\rceil,m+n-2i-1\}}^{n-i-1}
\nonumber\sum_{k=m-(i+2j-n+1)}^{m}(-1)^{i+2j-n+1}2^{n+2m-i-1}\\
&~&\cdot3^{i-j-k}{i \choose j}{j\choose n-i-j-1}{i-j\choose k}
{i+2j-n+1\choose m-k}\Bigg{)}x^m.
\end{eqnarray}
\end{theorem}

\begin{proof}
For convenience, let $F_n:=F(G_n,x)$.
Then the generating function of the sequence $\{F_n\}_{n=0}^{\infty}$ is
\begin{eqnarray}
\nonumber G(t)=\sum_{n=0}^{\infty}F_nt^n= F_0t^0+F_1t^1+F_2t^2+\sum_{n=3}^{\infty}F_nt^n.
\end{eqnarray}
By Theorem \ref{main1},
\begin{eqnarray}\label{m8}
\nonumber G(t)&=& F_0t^0+F_1t^1+F_2t^2+ (4x^2+3x)\sum_{n=3}^{\infty}F_{n-1}t^n-(8x^3+2x^2)\sum_{n=3}^{\infty}F_{n-2}t^n\\
\nonumber &~&+4x^3\sum_{n=3}^{\infty}F_{n-3}t^n\\
\nonumber &=&F_0t^0+F_1t^1+F_2t^2
 +(4x^2+3x)t(\sum_{n=0}^{\infty}F_{n}t^n-F_0t^0-F_1t^1)\\
\nonumber&~&
-(8x^3+2x^2)t^2(\sum_{n=0}^{\infty}F_{n}t^n-F_0t^0)+4x^3t^3\sum_{n=0}^{\infty}F_{n}t^n\\
\nonumber&=& 1-xt+((4x^2+3x)t-(8x^3+2x^2)t^2+4x^3t^3)G(t).
\end{eqnarray}
So
\begin{eqnarray}
\nonumber G(t)&=&\FS{1-xt}{1-((4x^2+3x)t-(8x^3+2x^2)t^2+4x^3t^3)}\\
\nonumber &=&(1-xt)\sum_{i=0}^{\infty}((4x^2+3x)t-(8x^3+2x^2)t^2+4x^3t^3)^i,
\end{eqnarray}
by the binomial theorem
\begin{eqnarray}
\nonumber  &~&\sum_{i=0}^{\infty}((4x^2+3x)t-(8x^3+2x^2)t^2+4x^3t^3)^i\\
\nonumber&=&\sum_{i=0}^{\infty}x^it^i\sum_{j=0}^{i}{i \choose j}
(4x^2t^2-(8x^2+2x)t)^j(4x+3)^{i-j}\\
\nonumber &=&\sum_{i=0}^{\infty}\sum_{j=0}^{i}\sum_{k=0}^{j}
(-1)^{j-k}2^{j+k}{i \choose j}{j\choose k}(4x+3)^{i-j}(4x+1)^{j-k}x^{i+j+k}t^{i+j+k}\\
\nonumber &=&\sum_{n=0}^{\infty}\sum_{i=\lceil\frac{n}{3}\rceil}^{n}\sum_{j=\lceil\frac{n-i}{2}\rceil}^{n-i}
(-1)^{i+2j-n}2^{n-i}{i \choose j}{j\choose n-i-j}(4x+3)^{i-j}(4x+1)^{i+2j-n}x^{n}t^{n}.
\end{eqnarray}
Similarly,
\begin{eqnarray}
\nonumber &~&xt\sum_{i=0}^{\infty}((4x^2+3x)t-(8x^3+2x^2)t^2+4x^3t^3)^i\\
\nonumber &=&\sum_{n=1}^{\infty}\sum_{i=\lceil\frac{n-1}{3}\rceil}^{n-1}
\sum_{j=\lceil\frac{n-i-1}{2}\rceil}^{n-i-1}
(-1)^{i+2j-n+1}2^{n-i-1}{i \choose j}{j\choose n-i-j-1}(4x+3)^{i-j}\\
\nonumber ~~~~~~~~~&~&\cdot(4x+1)^{i+2j-n+1}x^{n}t^{n}.
\end{eqnarray}

Therefore
\begin{eqnarray}
\nonumber &~&(1-xt)\sum_{i=0}^{\infty}((4x^2+3x)t-(8x^3+2x^2)t^2+4x^3t^3)^i\\
\nonumber &=&1+\sum_{n=1}^{\infty}x^n\Bigg{(}
\sum_{i=\lceil\frac{n}{3}\rceil}^{n}\sum_{j=\lceil\frac{n-i}{2}\rceil}^{n-i}
(-1)^{i+2j-n}2^{n-i}{i \choose j}{j\choose n-i-j}(4x+3)^{i-j}\\
\nonumber &~&\cdot(4x+1)^{i+2j-n}-\sum_{i=\lceil\frac{n-1}{3}\rceil}^{n-1}
\sum_{j=\lceil\frac{n-i-1}{2}\rceil}^{n-i-1}
(-1)^{i+2j-n+1}2^{n-i-1}{i \choose j}{j\choose n-i-j-1}\\
\nonumber ~~~~~~~~~&~&\cdot(4x+3)^{i-j}(4x+1)^{i+2j-n+1}\Bigg{)}t^{n}\\
\nonumber &=&1+\sum_{n=1}^{\infty}F_nt^n,
\end{eqnarray}
which implies that
\begin{eqnarray}\label{m9}
\nonumber F_n&=&x^n\Bigg{(}
\sum_{i=\lceil\frac{n}{3}\rceil}^{n}\sum_{j=\lceil\frac{n-i}{2}\rceil}^{n-i}
(-1)^{i+2j-n}2^{n-i}{i \choose j}{j\choose n-i-j}(4x+3)^{i-j}(4x+1)^{i+2j-n}\\
\nonumber &~&-\sum_{i=\lceil\frac{n-1}{3}\rceil}^{n-1}
\sum_{j=\lceil\frac{n-i-1}{2}\rceil}^{n-i-1}
(-1)^{i+2j-n+1}2^{n-i-1}{i \choose j}{j\choose n-i-j-1}(4x+3)^{i-j}\\
 &~&\cdot(4x+1)^{i+2j-n+1}\Bigg{)}.
\end{eqnarray}
Moreover,
\begin{eqnarray}
\nonumber &~&(4x+3)^{i-j}(4x+1)^{i+2j-n}\\
\nonumber &=&\sum_{m=0}^{2i+j-n}\sum_{k=m-(i+2j-n)}^{m}3^{i-j-k}4^{m}{i-j\choose k}{i+2j-n\choose m-k}x^{m},
\end{eqnarray}
thus
\begin{eqnarray}\label{m10}
\nonumber &~&\sum_{i=\lceil\frac{n}{3}\rceil}^{n}\sum_{j=\lceil\frac{n-i}{2}\rceil}^{n-i}
(-1)^{i+2j-n}2^{n-i}{i \choose j}{j\choose n-i-j}(4x+3)^{i-j}(4x+1)^{i+2j-n}\\
\nonumber&=&\sum_{m=0}^{n}\sum_{i=max\{\lceil\frac{n}{3}\rceil,m\}}^{n}
\sum_{j=max\{\lceil\frac{n-i}{2}\rceil,m+n-2i\}}^{n-i}
\sum_{k=m-(i+2j-n)}^{m}(-1)^{i+2j-n}2^{n+2m-i}3^{i-j-k}{i \choose j}\\
\nonumber&~&\cdot{j\choose n-i-j}{i-j\choose k}{i+2j-n\choose m-k}x^{m}\\
\nonumber&=&2^{2n}x^n+\sum_{m=0}^{n-1}\sum_{i=max\{\lceil\frac{n}{3}\rceil,m\}}^{n}
\sum_{j=max\{\lceil\frac{n-i}{2}\rceil,m+n-2i\}}^{n-i}
\sum_{k=m-(i+2j-n)}^{m}(-1)^{i+2j-n}2^{n+2m-i}3^{i-j-k}{i \choose j}\\
&~&\cdot{j\choose n-i-j}{i-j\choose k}{i+2j-n\choose m-k}x^{m}.
\end{eqnarray}
Similarly,
\begin{eqnarray}\label{m11}
\nonumber &~&\sum_{i=\lceil\frac{n-1}{3}\rceil}^{n-1}
\sum_{j=\lceil\frac{n-i-1}{2}\rceil}^{n-i-1}
(-1)^{i+2j-n+1}2^{n-i-1}{i \choose j}{j\choose n-i-j-1}(4x+3)^{i-j}(4x+1)^{i+2j-n+1}\\
\nonumber &=&\sum_{m=0}^{n-1}\sum_{i=max\{\lceil\frac{n-1}{3}\rceil,m\}}^{n-1}
\sum_{j=max\{\lceil\frac{n-i-1}{2}\rceil,m+n-2i-1\}}^{n-i-1}
\sum_{k=m-(i+2j-n+1)}^{m}(-1)^{i+2j-n+1}2^{n+2m-i-1}\\
&~&\cdot3^{i-j-k}{i \choose j}{j\choose n-i-j-1}{i-j\choose k}
{i+2j-n+1\choose m-k}x^m.
\end{eqnarray}
Finally, we can obtain Eq. (\ref{m7}) by substituting Eqs. (\ref{m10}) and (\ref{m11}) into Eq. (\ref{m9}).
\end{proof}

By Theorem \ref{main2}, $x^i$ is a term of
$F(G_n,x)$ for $i=n,n+1,\ldots,2n$, thus we have the following corollary.

\begin{corollary}\label{Spec}{\bf .}
\begin{enumerate}
  \item $f(G_n)=n$, $F(G_n)=2n$;
  \item Spec$_f(G_n)=[n,2n]$.
\end{enumerate}
\end{corollary}

In the following we calculate the degree of freedom of $G_n$,
and reveal its  asymptotic behavior with respect to $\Phi(G_n)$.
\begin{theorem}\label{main111}{\bf .}
\begin{eqnarray}\label{m222}
\nonumber IDF(G_n)=-4&+&\frac{50+22\sqrt{5}}{25}(3-\sqrt{5})^n+\frac{13-3\sqrt{5}}{10}n(3-\sqrt{5})^n\\
&+&\frac{50-22\sqrt{5}}{25}(3+\sqrt{5})^n+\frac{13+3\sqrt{5}}{10}n(3+\sqrt{5})^n.
\end{eqnarray}
\end{theorem}

\begin{proof}
According to Theorem \ref{main1},
\begin{eqnarray*}
\FS{d}{dx}F(G_n,x)&=&(8x+3)F(G_{n-1},x)+(4x^2+3x)\FS{d}{dx}F(G_{n-1},x)
-(24x^2+4x)F(G_{n-2},x)\\
&~&-(8x^3+2x^2)\FS{d}{dx}F(G_{n-2},x)+12x^2F(G_{n-3},x)+4x^3\FS{d}{dx}F(G_{n-3},x).
\end{eqnarray*}
For convenience, let $IDF_n:=IDF(G_n)$.
Then
\begin{eqnarray*}
IDF_{n}&=&\FS{d}{dx}F(G_n,x)\Big{|}_{x=1}\\
&=& 7IDF_{n-1}-10IDF_{n-2}+4IDF_{n-3}+11\Phi_{n-1}-28\Phi_{n-2}+12\Phi_{n-3}.
\end{eqnarray*}
By Lemma \ref{pm},
\begin{eqnarray*}
IDF_{n}
&=&7IDF_{n-1}-10IDF_{n-2}+4IDF_{n-3}+11\Phi_{n-1}-3(6\Phi_{n-2}-4\Phi_{n-3})
-10\Phi_{n-2}\\
&=&7IDF_{n-1}-10IDF_{n-2}+4IDF_{n-3}+8\Phi_{n-1}-10\Phi_{n-2}.
\end{eqnarray*}
Therefore
\begin{eqnarray*}
IDF_{n+1}=7IDF_{n}-10IDF_{n-1}+4IDF_{n-2}+8\Phi_{n}-10\Phi_{n-1},
\end{eqnarray*}
and
\begin{eqnarray*}
IDF_{n+2}
&=&7IDF_{n+1}-10IDF_{n}+4IDF_{n-1}+8\Phi_{n+1}-10\Phi_{n}.
\end{eqnarray*}
Note that
\begin{eqnarray*}
8\Phi_{n+1}-10\Phi_{n}
&=&8(6\Phi_{n}-4\Phi_{n-1})-10(6\Phi_{n-1}-4\Phi_{n-2})\\
&=&6(8\Phi_{n}-10\Phi_{n-1})-4(8\Phi_{n-1}-10\Phi_{n-2})\\
&=&6(7IDF_{n}-10IDF_{n-1}+4IDF_{n-2}+8\Phi_{n}-10\Phi_{n-1})\\
&~&-4(7IDF_{n-1}-10IDF_{n-2}+4IDF_{n-3}+8\Phi_{n-1}-10\Phi_{n-2})\\
&~&-6(7IDF_{n}-10IDF_{n-1}+4IDF_{n-2})\\
&~&+4(7IDF_{n-1}-10IDF_{n-2}+4IDF_{n-3})\\
&=&6IDF_{n+1}-46IDF_{n}+88IDF_{n-1}-64IDF_{n-2}+16IDF_{n-3},
\end{eqnarray*}
so
\begin{equation}\label{m223}
IDF_{n+2}=13IDF_{n+1}-56IDF_{n}+92IDF_{n-1}-64IDF_{n-2}+16IDF_{n-3}.
\end{equation}

The characteristics equation of Eq.
(\ref{m223}) is $x^5-13x^{4}+56x^3-92x^2+64x-16=0$,
whose roots are $x_0=1$, $x_1=x_2=3-\sqrt{5}$, $x_3=x_4=3+\sqrt{5}$.
Suppose the general solution of (\ref{m223})
is $IDF_n=\lambda_0+\lambda_1(3-\sqrt{5})^n+\lambda_2n(3-\sqrt{5})^n+
\lambda_3(3+\sqrt{5})^n+\lambda_4n(3+\sqrt{5})^n$.
According to the initial values $IDF_4=5948$,
$IDF_5=38908$, $IDF_6=244348$, $IDF_7=1492092$, and $IDF_8=8926204$,
we can obtain $\lambda_0=-4$, $\lambda_1=\frac{50+22\sqrt{5}}{25}$,
$\lambda_2=\frac{13-3\sqrt{5}}{10}$,
$\lambda_3=\frac{50-22\sqrt{5}}{25}$, and $\lambda_4=\frac{13+3\sqrt{5}}{10}$.
Therefore Eq. (\ref{m222}) holds for $n\geq4$.
In fact, (\ref{m222}) is true for $n=0,1,2,3$ as well,
the proof is completed.
\end{proof}

By Theorems \ref{PM} and \ref{main111},
we have the following corollary.

\begin{corollary}
\begin{equation*}
\lim\limits_{n\rightarrow\infty}\FS{IDF_n}{n\Phi_n}
=\FS{-5+6\sqrt{5}}{5}.
\end{equation*}
\end{corollary}

\section{Anti-forcing polynomials}

The \emph{anti-forcing polynomial}  \cite{Hwang, Zh2} of a graph $G$ is defined
as below:
\begin{equation}\label{eq101}
Af(G,x)=\sum_{M\in \mathcal{M}(G)} x^{af(G,M)}.
\end{equation}

Thus the following expression is immediate:

\begin{equation}\label{eq102}
\nonumber Af(G,x)=\sum_{i=af(G)}^{Af(G)} u(G,i)x^{i},
\end{equation}
where $u(G,i)$ is the number of perfect matchings
with the anti-forcing number $i$.

The anti-forcing polynomial can count the number of perfect matchings
with the same anti-forcing number,
that is, the distribution of all anti-forcing numbers is showed.
If $G$ is a null graph or a graph with a unique perfect matching,
then $Af(G,x)=1$.
Obviously,  $\Phi(G)=Af(G,1)$, and the sum of the anti-forcing
numbers of all perfect matchings is $\frac{d}{dx}Af(G,x)\big|_{x=1}$.

In the following, a recurrence formula of anti-forcing
polynomial of $G_n$ is showed.
First, some useful local lemmas on $H_n$ and $G_n$
are discussed.

\begin{lemma}\label{Huv}{\bf .}
Let $n$ be a positive integer, and $0\leq k\leq n-1$,
and let $\mathcal{M}_{2k+1}(H_n)=\{M\in\mathcal{M}(H_n)|u_{2k+1}v_{2k+1}\in M$,
and $u_{j}v_{j}\not\in M$ for $j<2k+1\}$.
Then
\begin{equation*}
\sum_{M\in\mathcal{M}_{2k+1}(H_n)}x^{af(H_n,M)}=(x^3+3x^2)^kx^2Af(G_{n-1-k},x).
\end{equation*}
\end{lemma}

\begin{proof}
Up to isomorphism,
$H_n$ can be split into two subsystems $H_k$ ($H_0$ is an empty graph) on the left and $H_{n-k}$ on the right
by deleting edges $u_{2k}u_{2k+1}$ and $v_{2k}v_{2k+1}$ from $H_n$ (see Fig. \ref{SQ2}).
Let $M\in \mathcal{M}_{2k+1}(H_n)$.
Then the restrictions $M^\prime$ and $M^{\prime\prime}$ of $M$ on
$H_k$ and $H_{n-k}$ are perfect matchings of $H_k$ and $H_{n-k}$, respectively.
Let $\mathcal{A}^\prime$ and $\mathcal{A}^{\prime\prime}$
be maximum non-crossing compatible $M^\prime$-alternating set and
$M^{\prime\prime}$-alternating set of $H_k$ and $H_{n-k}$
with $I(\mathcal{A}^\prime)$ and $I(\mathcal{A}^{\prime\prime})$ as small as possible, respectively.
Note that vertical edges $u_{j}v_{j}\not\in M$ for $j\leq 2k$,
by Remark 1, $\mathcal{A}=\mathcal{A}^\prime\cup\mathcal{A}^{\prime\prime}$
is a maximum non-crossing compatible $M$-alternating set of $H_n$.
According to Lemma \ref{minimax9},
\begin{equation}\label{afhn}
af(H_n,M)=|\mathcal{A}^\prime|+|\mathcal{A}^{\prime\prime}|=af(H_k,M^{\prime})+af(H_{n-k},M^{\prime\prime}).
\end{equation}

For $af(H_k,M^{\prime})$,
let $N_i$ be the substructure consisting of three squares
$s_{i,2}$, $s_{i,3}$ and $s_{i,4}$ $(i=1,2,\ldots,k)$.
Then the restriction of $M^\prime$ on $N_i$ is a perfect matching of $N_i$.
Recall that $u_{2i-1}v_{2i-1},u_{2i}v_{2i}\not\in M^\prime$,
there are two cases to be considered.
If $u_{2i-1}u_{2i},v_{2i-1}v_{2i}\in M^\prime$,
then $s_{i,2}$, $s_{i,3}$ and $s_{i,4}$ all are $M^\prime$-alternating squares.
By Remark 1,  $s_{i,2}$, $s_{i,3}$ and $s_{i,4}$  all belong to $\mathcal{A}^\prime$.
So the substructure $N_i$ contributes 3 to $af(H_k,M^{\prime})$.
If one of $u_{2i-1}u_{2i}$ and $v_{2i-1}v_{2i}$ is not in $M^\prime$,
then there are three possible cases such that only $s_{i,2}$ and $s_{i,4}$
are $M^\prime$-alternating squares.
Hence $N_i$ contributes 2 to $af(H_k,M^{\prime})$.
Suppose there are $r (0\leq r\leq k)$ substructures
$N_i$ contributing 3 to $af(H_k,M^{\prime})$,
then $af(H_k,M^{\prime})=3r+2(k-r)$.
Let $\mathcal{M}^{\prime}(H_k)=\{M^\prime\in\mathcal{M}(H_k)|u_iv_i\not\in M^{\prime},i=1,2,\ldots,2k\}$.
Then
\begin{equation}\label{afhn1}
\sum_{M^\prime\in\mathcal{M}^\prime(H_k)}x^{af(H_k,M^\prime)}
=\sum_{r=0}^{k}{k\choose r}3^{k-r}x^{2(k-r)}x^{3r}=(x^3+3x^2)^k.
\end{equation}

For $af(H_{n-k},M^{\prime\prime})$,
we recall that $u_{2k+1}v_{2k+1}\in M^{\prime\prime}$,
so there must be a vertical edge $u_{2m+2}v_{2m+2}\in M^{\prime\prime}(k\leq m\leq n-1)$ such that every vertical edge $u_{j}v_{j}(2k+1<j<2m+2)$
is not in $M^{\prime\prime}$.
As a consequence, the fragment of $H_{n-k}$ between the two vertical
edges $u_{2k+1}v_{2k+1}$ and $u_{2m+2}v_{2m+2}$ forms
the substructure $W_{m-k+1}$,
which also contains the substructure $L_{2(m-k)+1}$ (see Fig. \ref{LH}).
Note that $H_{n-k}\ominus\{u_{2k+1},v_{2k+1},w_{2k+1},w_{2k+2},z_{2k+1},z_{2k+2}\}$
is isomorphic to $G_{n-k-1}$,
and the restriction $M^{\prime\prime\prime}$ of $M^{\prime\prime}$ on $G_{n-k-1}$
is a perfect matching of $G_{n-k-1}$.
Let $\mathcal{A}^{\prime\prime\prime}$ be a maximum non-crossing
compatible $M^{\prime\prime\prime}$-alternating set of $G_{n-k-1}$
with $I(\mathcal{A}^{\prime\prime\prime})$ as small as possible,
and let $P(W_{m-k+1})$ and $P(L_{2(m-k)+1})$
be peripheries of $W_{m-k+1}$ and $L_{2(m-k)+1}$, respectively.
By Remark 1,
$\mathcal{A}^{\prime\prime\prime}\cup\{P(W_{m-k+1}),P(L_{2(m-k)+1})\}$
is a maximum non-crossing compatible $M^{\prime\prime}$-alternating set of
$H_{n-k}$.
By Lemma \ref{minimax9},
$af(H_{n-k},M^{\prime\prime})=2+|\mathcal{A}^{\prime\prime\prime}|=2+af(G_{n-k-1},M^{\prime\prime\prime})$. Let $\mathcal{M}^{\prime\prime}(H_{n-k})=\{M^{\prime\prime}\in \mathcal{M}(H_{n-k})| u_{2k+1}$
$v_{2k+1}\in M^{\prime\prime} \}$.
Then
\begin{eqnarray}\label{afhn2}
\sum_{M^{\prime\prime}\in\mathcal{M}^{\prime\prime}(H_{n-k})}x^{af(H_{n-k},M^{\prime\prime})}
=\sum_{M^{\prime\prime\prime}\in\mathcal{M}(G_{n-k-1})}x^{2+af(G_{n-k-1},M^{\prime\prime\prime})}=x^2Af(G_{n-k-1},x).
\end{eqnarray}
By Eqs. (\ref{afhn}), (\ref{afhn1}) and (\ref{afhn2}),
\begin{eqnarray*}
\sum_{M\in\mathcal{M}_{2k+1}(H_n)}x^{af(H_n,M)}&=&
\sum_{M^\prime\in\mathcal{M}^\prime(H_k),
M^{\prime\prime}\in\mathcal{M}^{\prime\prime}(H_{n-k})}x^{af(H_k,M^\prime)+af(H_{n-k},M^{\prime\prime})}\\
&=&\sum_{M^\prime\in\mathcal{M}^\prime(H_k)}x^{af(H_k,M^\prime)}
\sum_{M^{\prime\prime}\in\mathcal{M}^{\prime\prime}(H_{n-k})}x^{af(H_{n-k},M^{\prime\prime})}\\
&=&(x^3+3x^2)^kx^2Af(G_{n-k-1},x).
\end{eqnarray*}
\end{proof}

The following lemma is the case  of Eq. (\ref{afhn1}) for $k=n$.
\begin{lemma}\label{afhn0}{\bf .}
Let $\mathcal{M}_0(H_n)=\{M\in \mathcal{M}(H_n)|u_jv_j\not\in M$ for $j=1,2,\ldots,2n\}$.
Then
\begin{equation*}
\sum_{M\in \mathcal{M}_0(H_n)}x^{af(H_n,M)}=(x^3+3x^2)^n.
\end{equation*}
\end{lemma}

\begin{lemma}\label{Guv}{\bf .}
Let $n$ be a positive integer, and $0\leq k\leq n-1$,
$\mathcal{M}_{2k+1}^{u_0v_0}(G_n)=\{M\in\mathcal{M}(G_n)|u_0v_0, u_{2k+1}v_{2k+1}\in M,$ and $u_{j}v_{j}\not\in M$ for $j=1,2,\ldots,2k\}$.
Then
\begin{eqnarray*}
\sum_{M\in\mathcal{M}_{2k+1}^{u_0v_0}(G_n)}x^{af(G_n,M)}
=((x^3+3x^2)^k+(x-1)x^{3k})x^2Af(G_{n-k-1},x).
\end{eqnarray*}
\end{lemma}

\begin{proof}
Let $M\in \mathcal{M}_{2k+1}^{u_0v_0}(G_n)$.
Then $M^\prime=M\setminus\{u_0v_0\}$ is a perfect matching of $H_n=G_n\ominus\{u_0,v_0\}$,
and $M^\prime$ belongs to $\mathcal{M}_{2k+1}(H_n)$ (defined in Lemma \ref{Huv}).
Let $\mathcal{A}^\prime$ be a maximum non-crossing compatible
$M^\prime$-alternating set of $H_n$ with $I(\mathcal{A}^\prime)$ as small as possible.
There are two cases to be considered.
If $u_{2i+1}u_{2i+2}$ and $v_{2i+1}v_{2i+2}$ belong to $M$ for all
$i=0,1,\ldots,k-1$,
then there is a substructure $L_{2k+1}$ ($L_1$ is the square $s_{1,1}$) between the vertical
edges $u_0v_0$ and $u_{2k+1}v_{2k+1}$, whose periphery $P(L_{2k+1})$
is an $M$-alternating cycle (see Fig. \ref{SQ1}).
By Lemma \ref{minir} and Remark 1,
$\mathcal{A}^\prime\cup\{P(L_{2k+1})\}$ is a maximum compatible
$M$-alternating set of $G_n$.
By Lemma \ref{minimax9}, $af(G_n,M)=1+|\mathcal{A}^\prime|=1+af(H_n,M^\prime)$.
Suppose there exists a $j (0\leq j\leq k-1)$ such that  $u_{2j+1}u_{2j+2}$ or
$v_{2j+1}v_{2j+2}$ does not belong to $M$,
then $\mathcal{A}^\prime$ can be a maximum compatible
$M$-alternating set of $G_n$,
so $af(G_n,M)=af(H_n,M^\prime)$.
Analogously,
\begin{eqnarray*}
\sum_{M\in\mathcal{M}_{2k+1}^{u_0v_0}(G_n)}x^{af(G_n,M)}
&=&(\sum_{r=0}^{k-1}{k\choose r}3^{k-r}x^{2(k-r)}x^{3r}+x^{3k+1})x^2Af(G_{n-k-1},x)\\
&=&((x^3+3x^2)^k+(x-1)x^{3k})x^2Af(G_{n-k-1},x).
\end{eqnarray*}
\end{proof}

\begin{lemma}\label{Guv1}{\bf .}
Let $\mathcal{M}_0^{u_0v_0}(G_n)=\{M\in \mathcal{M}(G_n)|u_0v_0\in M, u_jv_j\not\in M$ for $j=1,2,\ldots,2n\}$.
Then
\begin{equation*}
\sum_{M\in \mathcal{M}_0^{u_0v_0}(G_n)}x^{af(G_n,M)}=(x^3+3x^2)^n.
\end{equation*}
\end{lemma}

\begin{proof}
Let $M\in\mathcal{M}_0^{u_0v_0}(G_n)$.
Then $M^\prime=M\setminus\{u_0v_0\}\in \mathcal{M}_0(H_n)$.
On the other hand,
if $M^\prime\in \mathcal{M}_0(H_n)$,
then $M=M^\prime\cup\{u_0v_0\}\in \mathcal{M}_0^{u_0v_0}(G_n)$.
So $|\mathcal{M}_0^{u_0v_0}(G_n)|=|\mathcal{M}_0(H_n)|$.
Let $\mathcal{A}^\prime$ be a maximum non-crossing compatible
$M^\prime$-alternating set of $H_n$ with $I(\mathcal{A}^\prime)$ as small as possible.
By Lemma \ref{minir} and Remark 1, $\mathcal{A}^\prime$ also is a maximum
compatible $M$-alternating set of $G_n$.
By Lemma \ref{minimax9},
$af(G_n,M)=af(H_n,M^\prime)$.
According to Lemma \ref{afhn0},
\begin{equation*}
\sum_{M\in \mathcal{M}_0^{u_0v_0}(G_n)}x^{af(G_n,M)}=\sum_{M^\prime\in \mathcal{M}_0(H_n)}x^{af(H_n,M^\prime)} =(x^3+3x^2)^n.
\end{equation*}
\end{proof}

\begin{lemma}\label{Guv2}{\bf .}
Let $\mathcal{M}_{\bar{0}}(G_n)=\{M\in \mathcal{M}(G_n)|u_0v_0\not\in M\}$.
Then
\begin{equation*}
\sum_{M\in \mathcal{M}_{\bar{0}}(G_n)}x^{af(G_n,M)}=xAf(G_{n-1},x).
\end{equation*}
\end{lemma}

\begin{proof}
Let $M\in \mathcal{M}_{\bar{0}}(G_n)$.
Then $u_0u_1,v_0v_1,w_1w_2,z_1z_2\in M$,
$s_{1,1}$ is an $M$-alternating square,
and the restriction $M^\prime$ of $M$
on $G_{n-1}=G_{n}\ominus\{u_0,u_1,v_0,v_1,w_1,w_2,z_1,z_2\}$
is a perfect matching of $G_{n-1}$, see Fig. \ref{SQ1}.
Let $\mathcal{A}^\prime$ be a maximum non-crossing
compatible $M^\prime$-alternating set of $G_{n-1}$ with $I(\mathcal{A}^\prime)$ as small as possible.
By Lemma \ref{minir},
$\mathcal{A}^\prime\cup\{s_{1,1}\}$ is a maximum
compatible $M$-alternating set of $G_n$.
By Lemma \ref{minimax9},
$af(G_n,M)=1+af(G_{n-1},M^\prime)$.
Therefore
\begin{equation*}
\sum_{M\in \mathcal{M}_{\bar{0}}(G_n)}x^{af(G_n,M)}=
\sum_{M^\prime\in \mathcal{M}(G_{n-1})}x^{1+af(G_{n-1},M^\prime)}=xAf(G_{n-1},x).
\end{equation*}
\end{proof}

\begin{lemma}\label{hnr}{\bf .}
Let $n$ be a positive integer.
Then
\begin{equation*}
Af(H_n,x)=x^2Af(G_{n-1},x)+(x^3+3x^2)Af(H_{n-1},x).
\end{equation*}
\end{lemma}

\begin{proof}
We can divide $\mathcal{M}(H_n)$ into two subsets:
$\mathcal{M}_{1}(H_n)=\{M\in \mathcal{M}(H_n)|u_1v_1\in M\}$
and $\mathcal{M}_{\bar{1}}(H_n)=\{M\in \mathcal{M}(H_n)|u_1v_1\not\in M\}$
(see Fig. \ref{SQ2}).
In fact, $\mathcal{M}_{1}(H_n)$ is the case of Lemma \ref{Huv} for $k=0$.
So
\begin{equation*}
\sum_{M\in\mathcal{M}_1(H_n)}x^{af(H_n,M)}=x^2Af(G_{n-1},x).
\end{equation*}

Let $M\in\mathcal{M}_{\bar{1}}(H_n)$.
Then $s_{1,2}$ and $s_{1,4}$ have to be $M$-alternating squares,
and the restriction of $M$ on $H_{n-1}=H_n\ominus\{w_1,w_2,u_1,u_2,v_1,v_2,z_1,z_2\}$
is a perfect matching of $H_{n-1}$.
Similar to Lemma \ref{Huv},
\begin{equation*}
\sum_{M\in\mathcal{M}_{\bar{1}}(H_n)}x^{af(H_{n},M)}=(x^3+3x^2)Af(H_{n-1},x).
\end{equation*}
By Eq. (\ref{eq101}),
\begin{eqnarray*}
Af(H_n,x)=\sum_{M\in\mathcal{M}(H_n)}x^{af(H_n,M)}
&=&\sum_{M\in\mathcal{M}_1(H_n)}x^{af(H_n,M)}+
\sum_{M\in\mathcal{M}_{\bar{1}}(H_n)}x^{af(H_n,M)}\\
&=&x^2Af(G_{n-1},x)+(x^3+3x^2)Af(H_{n-1},x)
\end{eqnarray*}\end{proof}

\begin{theorem}\label{mainA}{\bf .}
For $n\geq3$,
\begin{eqnarray*}
\nonumber Af(G_n,x)&=&(3x^3+3x^2+x)Af(G_{n-1},x)-(2x^6+6x^5-x^4+3x^3)Af(G_{n-2},x)\\
&~&+(x^7+3x^6)Af(G_{n-3},x),
\end{eqnarray*}
where $Af(G_0,x)=1$, $Af(G_1,x)=2x^3+3x^2+x$ and $Af(G_2,x)=4x^6+9x^5+15x^4+3x^3+x^2$.
\end{theorem}

\begin{proof}
By Lemmas \ref{Guv}, \ref{Guv1} and \ref{Guv2},
$\mathcal{M}(G_n)$ can be divided into $n+2$ subsets:
$\mathcal{M}_{0}^{u_0v_0}(G_n)$, $\mathcal{M}_{\bar{0}}(G_n)$ and $\mathcal{M}_{2k+1}^{u_0v_0}(G_n),k=0,1,2,\ldots,n-1$.
By Eq. (\ref{eq101}),
\begin{eqnarray}\label{AFGx}
\nonumber Af(G_n,x)&=&\sum_{M\in\mathcal{M}(G_n)}x^{af(G_n,M)}\\
\nonumber &=& \sum_{k=0}^{n-1}\sum_{M\in\mathcal{M}_{2k+1}^{u_0v_0}(G_n)}x^{af(G_n,M)}+
\sum_{M\in \mathcal{M}_{\bar{0}}(G_n)}x^{af(G_n,M)}+\sum_{M\in \mathcal{M}_0^{u_0v_0}(G_n)}x^{af(G_n,M)}\\
\nonumber&=&\sum_{k=0}^{n-1}(x^3+3x^2)^kx^2Af(G_{n-k-1},x)+
x^2(x-1)\sum_{k=0}^{n-1}x^{3k}Af(G_{n-k-1},x)\\
&~&+xAf(G_{n-1},x)+(x^3+3x^2)^n.
\end{eqnarray}

By Lemmas \ref{Huv} and \ref{afhn0},
\begin{eqnarray}\label{AFHx}
\nonumber Af(H_n,x)&=&\sum_{M\in\mathcal{M}(H_n)}x^{af(H_n,M)}\\
\nonumber &=&\sum_{k=0}^{n-1}\sum_{M\in\mathcal{M}_{2k+1}(H_n)}x^{af(H_n,M)}
+\sum_{M\in \mathcal{M}_0(H_n)}x^{af(H_n,M)}\\
&=&\sum_{k=0}^{n-1}(x^3+3x^2)^kx^2Af(G_{n-k-1},x)+(x^3+3x^2)^n.
\end{eqnarray}
Substituting Eq. (\ref{AFHx}) into (\ref{AFGx}),
then
\begin{eqnarray*}
Af(G_n,x)=Af(H_n,x)+xAf(G_{n-1},x)
+x^2(x-1)\sum_{k=0}^{n-1}x^{3k}Af(G_{n-k-1},x).
\end{eqnarray*}
So
\begin{eqnarray}\label{x3}
\nonumber x^3Af(G_n,x)&=&x^3Af(H_n,x)+x^4Af(G_{n-1},x)\\
&~&+x^2(x-1)\sum_{k=0}^{n-1}x^{3(k+1)}Af(G_{n-k-1},x)
\end{eqnarray}
and
\begin{eqnarray}\label{afgn}
Af(G_{n+1},x)=Af(H_{n+1},x)+xAf(G_{n},x)
+x^2(x-1)\sum_{k=0}^{n}x^{3k}Af(G_{n-k},x).
\end{eqnarray}
Subtracting  Eq. (\ref{x3}) from (\ref{afgn}),
\begin{eqnarray}\label{afgn1}
\nonumber Af(G_{n+1},x)&=&(2x^3-x^2+x)Af(G_{n},x)-x^4Af(G_{n-1},x)\\
&~&+Af(H_{n+1},x)-x^3Af(H_{n},x).
\end{eqnarray}

According to Lemma \ref{hnr},
\begin{equation*}
Af(H_{n+1},x)-x^3Af(H_{n},x)=x^2Af(G_{n},x)+3x^2Af(H_{n},x),
\end{equation*}
therefore
\begin{eqnarray*}\label{afgn2}
\nonumber Af(G_{n+1},x)&=&(2x^3-x^2+x)Af(G_{n},x)-x^4Af(G_{n-1},x)\\
&~&+x^2Af(G_{n},x)+3x^2Af(H_{n},x)\\
&=&(2x^3+x)Af(G_{n},x)-x^4Af(G_{n-1},x)+3x^2Af(H_{n},x).
\end{eqnarray*}
This implies that
\begin{eqnarray}\label{afgn3}
 3x^2Af(H_{n},x)&=&Af(G_{n+1},x)-(2x^3+x)Af(G_{n},x)+x^4Af(G_{n-1},x),
\end{eqnarray}
and
\begin{eqnarray}\label{afgn4}
3x^2Af(H_{n-1},x)&=&Af(G_{n},x)-(2x^3+x)Af(G_{n-1},x)+x^4Af(G_{n-2},x).
\end{eqnarray}
By Lemma \ref{hnr},
\begin{equation}\label{hnx3}
3x^2Af(H_n,x)=3x^4Af(G_{n-1})+(x^3+3x^2)3x^2Af(H_{n-1},x).
\end{equation}
Substituting Eqs. (\ref{afgn3}) and (\ref{afgn4}) into (\ref{hnx3}),
we can obtain the following recurrence formula
\begin{eqnarray*}
Af(G_{n+1},x)&=&(3x^3+3x^2+x)Af(G_{n},x)-(2x^6+6x^5-x^4+3x^3)Af(G_{n-1},x)\\
&~&+(x^7+3x^6)Af(G_{n-2},x),
\end{eqnarray*}
the proof is completed.
\end{proof}

According to Theorem \ref{mainA}, we can obtain the following
expression.

\begin{theorem}\label{mainB}{\bf .}
\begin{equation}\label{Afgnx}
Af(G_n,x)=R_n(x)+Q_n(x)+3x^{n+1}+x^n,
\end{equation}
where $R_n(x)$ and $Q_n(x)$ are listed in the Appendix.
\end{theorem}

\begin{proof}
Let $Af_n:=Af(G_n,x)$.
Then the generating function of the sequence $\{Af_n\}_{n=0}^\infty$ is
\begin{eqnarray}
\nonumber P(t)=\sum_{n=0}^{\infty}Af_nt^n=1+Af_1t+Af_2t^2+\sum_{n=3}^{\infty}Af_nt^n.
\end{eqnarray}
By Theorem \ref{mainA},
\begin{eqnarray}
\nonumber P(t)&=& 1+Af_1t+Af_2t^2+(3x^3+3x^2+x)\sum_{n=3}^{\infty}Af_{n-1}t^n\\
\nonumber &~&-(2x^6+6x^5-x^4+3x^3)\sum_{n=3}^{\infty}Af_{n-2}t^n
+(x^7+3x^6)\sum_{n=3}^{\infty}Af_{n-3}t^n\\
\nonumber &=& 1-x^3t+((3x^3+3x^2+x)t
-(2x^6+6x^5-x^4+3x^3)t^2+(x^7+3x^6)t^3)P(t).
\end{eqnarray}
So
\begin{eqnarray}
\nonumber P(t)&=&\FS{1-x^3t}{1-((3x^3+3x^2+x)t
-(2x^6+6x^5-x^4+3x^3)t^2+(x^7+3x^6)t^3)}\\
\nonumber &=&(1-x^3t)\sum_{i=0}^{\infty}((3x^3+3x^2+x)t
-(2x^6+6x^5-x^4+3x^3)t^2+(x^7+3x^6)t^3)^i\\
\nonumber &=&(1-x^3t)
\sum_{i=0}^{\infty}\sum_{j=0}^{i}\sum_{k=0}^{i-j}
(-1)^{i-j-k}{i\choose j}{i-j\choose k}(3x^3+3x^2+x)^j(x^7+3x^6)^k\\
\nonumber &~&\cdot(2x^6+6x^5-x^4+3x^3)^{i-j-k}t^{2i-j+k}.
\end{eqnarray}
Let $n=2i-j+k$.
Then
\begin{eqnarray}
\nonumber P(t) &=& 1+
\sum_{n=1}^{\infty}\Bigg{(}
\sum_{i=\lceil\frac{n}{3}\rceil}^{n}
\sum_{j=max\{0,2i-n\}}^{\lfloor\frac{3i-n}{2}\rfloor}(-1)^{3i-2j-n}
{i\choose j}{i-j\choose n-2i+j}(3x^3+3x^2+x)^j\\
\nonumber &~&\cdot(x^7+3x^6)^{n-2i+j}(2x^6+6x^5-x^4+3x^3)^{3i-2j-n}+\\
\nonumber &~& \sum_{i=\lceil\frac{n-1}{3}\rceil}^{n-1}
\sum_{j=max\{0,2i-n+1\}}^{\lfloor\frac{3i-n+1}{2}\rfloor}
(-1)^{3i-2j-n+2}{i\choose j}{i-j\choose n-2i+j-1}x^3(3x^3+3x^2+x)^j\\
\nonumber &~&\cdot(x^7+3x^6)^{n-2i+j-1}(2x^6+6x^5-x^4+3x^3)^{3i-2j-n+1}
\Bigg{)}t^{n}\\
\nonumber &=& 1+\sum_{n=1}^{\infty}Af_nt^n.
\end{eqnarray}
Therefore,
\begin{eqnarray*}
Af_n&=&\sum_{i=\lceil\frac{n}{3}\rceil}^{n}
\sum_{j=max\{0,2i-n\}}^{\lfloor\frac{3i-n}{2}\rfloor}(-1)^{3i-2j-n}
{i\choose j}{i-j\choose n-2i+j}(3x^3+3x^2+x)^j\\
\nonumber &~&\cdot(x^7+3x^6)^{n-2i+j}(2x^6+6x^5-x^4+3x^3)^{3i-2j-n}+\\
\nonumber &~& \sum_{i=\lceil\frac{n-1}{3}\rceil}^{n-1}
\sum_{j=max\{0,2i-n+1\}}^{\lfloor\frac{3i-n+1}{2}\rfloor}
(-1)^{3i-2j-n+2}{i\choose j}{i-j\choose n-2i+j-1}x^3(3x^3+3x^2+x)^j\\
 &~&\cdot(x^7+3x^6)^{n-2i+j-1}(2x^6+6x^5-x^4+3x^3)^{3i-2j-n+1}.
\end{eqnarray*}

Furthermore, we can derive the Eq. (\ref{Afgnx}) by using the binomial theorem,
more details are listed in the Appendix.
\end{proof}

By Theorem \ref{mainB}, $x^i$ is a term of
$Af(G_n,x)$ for $i=n,n+1,\ldots,3n$, thus we have the following corollary.

\begin{corollary}\label{Speca}{\bf .}
\begin{enumerate}
  \item $af(G_n)=n$, $Af(G_n)=3n$;
  \item Spec$_{af}(G_n)=[n,3n]$.
\end{enumerate}
\end{corollary}

By Theorem \ref{mainA},
we can obtain the exact expression of
the sum over all anti-forcing number of perfect matchings of $G_n$.

\begin{theorem}\label{EXPR}{\bf .}
\begin{eqnarray}\label{afgn5}
\nonumber \FS{d}{dx}Af(G_n,x)\Big{|}_{x=1}&=&-3+\FS{150+67\sqrt{5}}{100}(3-\sqrt{5})^n+
\FS{29-10\sqrt{5}}{20}n(3-\sqrt{5})^n\\
&~&+\FS{150-67\sqrt{5}}{100}(3+\sqrt{5})^n+
\FS{29+10\sqrt{5}}{20}n(3+\sqrt{5})^n.
\end{eqnarray}
\end{theorem}

\begin{proof}
Recall that $\Phi(G_n)=Af(G_n,1)$,
for convenience, let $\Phi_n:=\Phi(G_n)$ and $A_n:=\FS{d}{dx}Af(G_n,x)\Big{|}_{x=1}$.
By Theorem \ref{mainA},
\begin{eqnarray*}
A_n=7A_{n-1}-10A_{n-2}+4A_{n-3}+16\Phi_{n-1}-47\Phi_{n-2}+25\Phi_{n-3}.
\end{eqnarray*}
According to Lemma \ref{pm},
\begin{eqnarray*}
4A_n &=&28A_{n-1}-40A_{n-2}+16A_{n-3}+64\Phi_{n-1}-
188\Phi_{n-2}+100\Phi_{n-3}\\
&=&28A_{n-1}-40A_{n-2}+16A_{n-3}+64\Phi_{n-1}-
25(6\Phi_{n-2}-4\Phi_{n-3})-38\Phi_{n-2}\\
&=&28A_{n-1}-40A_{n-2}+16A_{n-3}+39\Phi_{n-1}-38\Phi_{n-2}.
\end{eqnarray*}
So
\begin{equation*}
4A_{n+1}=28A_{n}-40A_{n-1}+16A_{n-2}+39\Phi_{n}-38\Phi_{n-1}
\end{equation*}
and
\begin{equation}\label{An2}
4A_{n+2}=28A_{n+1}-40A_{n}+16A_{n-1}+39\Phi_{n+1}-38\Phi_{n}.
\end{equation}
Note that
\begin{eqnarray*}
39\Phi_{n+1}-38\Phi_{n}&=&
39(6\Phi_{n}-4\Phi_{n-1})-38(6\Phi_{n-1}-4\Phi_{n-2})\\
&=&6(39\Phi_{n}-38\Phi_{n-1})-4(39\Phi_{n-1}-38\Phi_{n-2})\\
&=&6(28A_{n}-40A_{n-1}+16A_{n-2}+39\Phi_{n}-38\Phi_{n-1})\\
&~&-4(28A_{n-1}-40A_{n-2}+16A_{n-3}+39\Phi_{n-1}-38\Phi_{n-2})\\
&~&-6(28A_{n}-40A_{n-1}+16A_{n-2})+4(28A_{n-1}-40A_{n-2}+16A_{n-3})\\
&=& 24A_{n+1}-184A_{n}+352A_{n-1}-256A_{n-2}+64A_{n-3},
\end{eqnarray*}
substituting it into Eq. (\ref{An2}),
\begin{equation*}
A_{n+2}=13A_{n+1}-56A_{n}+92A_{n-1}-64A_{n-2}+16A_{n-3}.
\end{equation*}

This recurrence relation is the same as (\ref{m223}),
by initial values $A_4=7721$,
$A_5=50541$, $A_6=317565$, $A_7=1939901$ and $A_8=11608381$,
we can prove Eq. (\ref{afgn5}) for $n\geq4$.
Actually, Eq. (\ref{afgn5}) also holds for $n=0,1,2,3$,
so the proof is completed.
\end{proof}

By Theorems \ref{PM} and \ref{EXPR}, we can obtain the following asymptotic behavior.

\begin{corollary}{\bf .}
\begin{equation*}
\lim\limits_{n\rightarrow\infty}\FS{A_n}{n\Phi_n}=\FS{5+37\sqrt{5}}{40}.
\end{equation*}
\end{corollary}

\newpage
\thispagestyle{empty}
\newgeometry{left=1cm,right=1cm,top=1cm,bottom=1cm}
\begin{landscape}

\noindent \textbf{Appendix}
\vspace{1mm}
\begin{eqnarray*}\label{Afgnxx}
&~&\sum_{i=\lceil\frac{n}{3}\rceil}^{n}
\sum_{j=max\{0,2i-n\}}^{\lfloor\frac{3i-n}{2}\rfloor}(-1)^{3i-2j-n}
{i\choose j}{i-j\choose n-2i+j}(3x^3+3x^2+x)^j
(x^7+3x^6)^{n-2i+j}(2x^6+6x^5-x^4+3x^3)^{3i-2j-n}\\
&=&x^n\sum\limits_{q=0}^{2n}
\sum\limits_{i=max\{n-q,\lceil\frac{n}{3}\rceil\}}^{n}
\sum\limits_{j=max\{0,2i-n\}}^{min\{q+3i-2n,\lfloor\frac{4i-q}{2}\rfloor,
\lfloor\frac{3i-n}{2}\rfloor\}}
\sum\limits_{k=max\{0,\lceil\frac{q+4j-4i}{2}\rceil\}}^{min\{j,q+3i-j-2n\}}
\sum\limits_{r=max\{0,q+4j-k-4i\}}^{min\{k,q+3i-k-j-2n\}}
\sum\limits_{m=max\{0,q+5j+n-k-r-6i\}}^{min\{\alpha,q+3i-k-r-j-2n\}}
\sum\limits_{s=max\{0,\gamma+n-j-q\}}^
{min\{\beta,\lfloor\frac{\gamma+6i-5j-n-q}{3}\rfloor\}}\\
\nonumber &~&\sum\limits_{l=max\{0,\gamma+3i-3j-2s-q\}}^
{min\{\beta-s,\lfloor\frac{\gamma+6i-5j-3s-n-q}{2}\rfloor\}}{i\choose j}{i-j\choose \alpha}{j\choose k}{k\choose r}
{\alpha\choose m}{\beta\choose s}{\beta-s\choose l}{\beta-s-l\choose \theta}
(-1)^{\beta+l}~2^{\beta-s-l}~3^{k+i-j-m-l-\theta}x^q\\
&=&x^n\sum\limits_{q=2}^{2n}
\sum\limits_{i=max\{n-q,\lceil\frac{n}{3}\rceil\}}^{n}
\sum\limits_{j=max\{0,2i-n\}}^{min\{q+3i-2n,\lfloor\frac{4i-q}{2}\rfloor,
\lfloor\frac{3i-n}{2}\rfloor\}}
\sum\limits_{k=max\{0,\lceil\frac{q+4j-4i}{2}\rceil\}}^{min\{j,q+3i-j-2n\}}
\sum\limits_{r=max\{0,q+4j-k-4i\}}^{min\{k,q+3i-k-j-2n\}}
\sum\limits_{m=max\{0,q+5j+n-k-r-6i\}}^{min\{\alpha,q+3i-k-r-j-2n\}}
\sum\limits_{s=max\{0,\gamma+n-j-q\}}^
{min\{\beta,\lfloor\frac{\gamma+6i-5j-n-q}{3}\rfloor\}}\\
\nonumber &~&\sum\limits_{l=max\{0,\gamma+3i-3j-2s-q\}}^
{min\{\beta-s,\lfloor\frac{\gamma+6i-5j-3s-n-q}{2}\rfloor\}}{i\choose j}{i-j\choose \alpha}{j\choose k}{k\choose r}
{\alpha\choose m}{\beta\choose s}{\beta-s\choose l}{\beta-s-l\choose \theta}
(-1)^{\beta+l}~2^{\beta-s-l}~3^{k+i-j-m-l-\theta}x^q+3x^{n+1}+x^n\\
&=&R_n(x)+3x^{n+1}+x^n\\
&~&\\
&~&\sum_{i=\lceil\frac{n-1}{3}\rceil}^{n-1}
\sum_{j=max\{0,2i-n+1\}}^{\lfloor\frac{3i-n+1}{2}\rfloor}
(-1)^{3i-2j-n+2}{i\choose j}{i-j\choose n-2i+j-1}x^3(3x^3+3x^2+x)^j(x^7+3x^6)^{n-2i+j-1}
(2x^6+6x^5-x^4+3x^3)^{3i-2j-n+1}\\
&=&x^n\sum\limits_{q=2}^{2n}\sum\limits_{i=max\{n-q+1,\lceil\frac{n-1}{3}\rceil\}}^{n-1}
\sum\limits_{j=max\{0,2i-n+1\}}^
{min\{q+3i-2n,\lfloor\frac{4i-q+2}{2}\rfloor,\lfloor\frac{3i-n+1}{2}\rfloor\}}
\sum\limits_{k=max\{0,\lceil\frac{q+4j-4i-2}{2}\rceil\}}^{min\{j,q+3i-j-2n\}}
\sum\limits_{r=max\{0,q+4j-k-4i-2\}}^{min\{k,q+3i-k-j-2n\}}
\sum\limits_{m=max\{0,q+5j+n-k-r-6i-3\}}^{min\{\alpha-1,q+3i-k-r-j-2n\}}
\sum\limits_{s=max\{0,\gamma+n-j-q+1\}}^
{min\{\beta+1,\lfloor\frac{\gamma+6i-5j-n-q+3}{3}\rfloor\}}\\
 &~&\sum\limits_{l=max\{0,\gamma+3i-3j-2s-q+2\}}^
{min\{\beta-s+1,\lfloor\frac{\gamma+6i-5j-3s-n-q+3}{2}\rfloor\}}{i\choose j}{i-j\choose \alpha-1}{j\choose k}{k\choose r}{\alpha-1\choose m}{\beta+1\choose s}{\beta-s+1\choose l}
{\beta-s-l+1\choose \theta-2}(-1)^{\beta+l+2}~
2^{\beta-s-l+1}~3^{k+i-j-m-l-\theta+2}x^q\\
&=&Q_n(x)
\end{eqnarray*}

Where $n\geq1$, $\alpha=n-2i+j$, $\beta=3i-2j-n$, $\gamma=k+m+r$,
$\theta=q+3j+2s+l-k-r-m-3i$.
\end{landscape}


\begin{thebibliography}{99}
\setlength{\itemsep}{-1mm}


\bibitem{Adams} P. Adams, M. Mahdian, E.S. Mahmoodian, On the forced matching numbers of bipartite graphs, Discrete Math.
281 (2004) 1--12.
\bibitem{Afshani} P. Afshani, H. Hatami, E.S. Mahmoodian, On the spectrum of the forced matching number of graphs, Australas. J. Combin. 30 (2004) 147--160.
\bibitem{Berge} C. Berge, C. Chen, V. Chv\'{a}tal, C.S. Seow, Combinatorial properties of polyominoes, Combinatorica  1 (1981) 217--224.
\bibitem{Che} Z. Che, Z. Chen, Forcing on perfect matchings-A survey,
MATCH Commun. Math. Comput. Chem. 66 (2011) 93--136.
\bibitem{Che1} Z. Che, Z. Chen, Conjugated circuits and forcing edges, MATCH Commun. Math. Comput. Chem. 69 (2013) 721--731.
\bibitem{Cockayne} E.J. Cockayne, Chessboard domination problems, Discrete Math. 86 (1990) 13--20.
\bibitem{Cyvin} S. J. Cyvin, I. Gutman, Kekul\'{e} structures in benzenoid hydrocarbons (Lecture notes in chemistry 46), Springer Verlag, Berlin, 1988.
\bibitem{Deng1} H. Deng, The anti-forcing number of hexagonal chains, MATCH Commun. Math. Comput. Chem.  58 (2007) 675--682.
\bibitem{Deng2} H. Deng, The anti-forcing number of double hexagonal chains, MATCH Commun. Math. Comput. Chem.  60 (2008) 183--192.
\bibitem{ZhangD2}  K. Deng, H. Zhang, Anti-forcing spectra of perfect matchings of graphs, J. Comb. Optim. 33 (2017) 660--680.
\bibitem{ZhangD3} K. Deng, H. Zhang, Anti-forcing spectrum of any cata-condensed hexagonal system is continuous, Front. Math. China 12 (2017) 19--33.
\bibitem{ZhangD1}   K. Deng, H. Zhang, Extremal anti-forcing numbers of perfect matchings of graphs, Discrete Appl. Math. 224 (2017), 69--79.
\bibitem{Diwan} A.A. Diwan, The minimum forcing number of perfect matchings
in the hypercube, Discrete Math. 342 (2019) 1060--1062.
\bibitem{Harary} F. Harary, D. Klein, T. \v{Z}ivkovi\'{c}, Graphical properties of polyhexes: perfect matching vector and forcing, J. Math. Chem. 6 (1991) 295--306.
\bibitem{Hwang} H.K. Hwang, H. Lei, Y. Yeh, H. Zhang, Distribution of forcing and anti-forcing numbers of random perfect matchings on hexagonal chains and crowns (preprint, 2015). http://140.109.74.92/hk/?p=873
\bibitem{Jiang} X. Jiang, H. Zhang, On forcing matching number of boron-nitrogen fullerene graphs, Discrete Appl. Math.
 159 (2011) 1581--1593.
\bibitem{Jiang1} X. Jiang, H. Zhang, The maximum forcing number of cylindrical grid, toroidal 4-8 lattice and Klein bottle 4-8 lattice, J. Math. Chem. 54 (2016) 18--32.
\bibitem{Kasteleyn} P.W. Kasteleyn, The statistics of dimers on a lattice \uppercase\expandafter{\romannumeral1}:
    the number of dimer arrangements on a quadratic lattice, Phys. 27 (1961), 1209--1225.
\bibitem{Klein} D. Klein, M. Randi\'{c}, Innate degree of freedom of a graph, J. Comput. Chem. 8 (1987) 516--521.
\bibitem{Kleinerman} S. Kleinerman, Bounds on the forcing numbers of bipartite graphs, Discrete Math. 306 (2006) 66--73.
\bibitem{Lam} F. Lam, L. Pachter, Forcing number for stop signs, Theor. Comput. Sci. 303 (2003) 409--416.
\bibitem{Zh2} H. Lei, Y. Yeh, H. Zhang, Anti-forcing numbers of perfect matchings of graphs, Discrete Appl. Math. 202 (2016) 95--105.
\bibitem{Li2} X. Li, Hexagonal systems with forcing single edges,
Discrete Appl. Math. 72 (1997) 295--301.
\bibitem{LP}  L. Lov\'{a}sz, M. Plummer, Matching Theory, Annals of Discrete Mathematics, Vol 29.
Amsterdam: North-Holland, 1986.
\bibitem{Younger} C.L. Lucchesi, D.H. Younger, A minimax theorem for directed graphs, J. Lond. Math. Soc. 17 (1978) 369--374.
\bibitem{Motoyama} A. Motoyama, H. Hosoya, King and domino polyominals for polyomino graphs, J. Math. Phys. 18 (1997) 1485--1490.

\bibitem{Pachter} L. Pachter, P. Kim, Forcing matchings on square grids, Discrete Math. 190 (1998) 287--294.
\bibitem{Randic} M. Randi\'{c}, D. Klein,
in \emph{Mathematical and Computational Concepts in Chemistry}, ed. by N. Trinajsti\'{c} (Wiley, New York, 1985), pp. 274--282.
\bibitem{c70} M. Randi\'{c}, D. Vuki\v{c}evi\'{c}, Kekul\'{e} structures of fullerene $C_{70}$, Croat. Chem. Acta 79 (2006) 471--481.
\bibitem{Matthew} M.E. Riddle, The minimum forcing number for the torus and hypercube, Discrete Math. 245 (2002) 283--292.
\bibitem{Shi} L. Shi, H. Zhang,   Forcing and anti-forcing numbers of (3,6)-fullerenes, MATCH Commun. Math. Comput. Chem.  76 (2016) 597--614.
\bibitem{Shi1} L. Shi, H. Wang, H. Zhang, On the maximum forcing and anti-forcing numbers of (4,6)-fullerenes, Discrete Appl. Math. 233 (2017) 187--194.
\bibitem{Shi2}  L. Shi, H. Zhang, Tight upper bound on the maximum anti-forcing numbers of graphs, Discrete Math. Theor. Comput. Sci. 19(3) (2017) \#9.
\bibitem{c72} D. Vuki\v{c}evi\'{c}, I. Gutman, M. Randi\'{c}, On instability of fullerene $C_{72}$, Croat. Chem. Acta 79 (2006) 429--436.
\bibitem{c60} D. Vuki\v{c}evi\'{c}, M. Randi\'{c}, On Kekukl\'{e} structures of buckminsterfullerene, Chem. Phys. Lett. 401 (2005) 446--450.
\bibitem{VT1} D. Vuki\v{c}evi\'{c}, N. Trinajsti\'{e}, On the anti-forcing number of benzenoids, J. Math. Chem. 42 (2007) 575--583.
\bibitem{VT2} D. Vuki\v{c}evi\'{c}, N. Trinajsti\'{e}, On the anti-Kekul\'{e} number and anti-forcing number of cata-condensed bezenoids, J. Math. Chem. 43 (2008) 719--726.
\bibitem{Wang} H. Wang, D. Ye, H. Zhang, The forcing number of toroidal polyhexes, J. Math. Chem. 43 (2008) 457--475.
\bibitem{ZhangF} L. Xu, H. Bian, F. Zhang, Maximum forcing number of hexagonal systems, MATCH Commun. Math. Comput. Chem.  70 (2013) 493--500.
\bibitem{Yang}  Q. Yang, H. Zhang, Y. Lin, On the anti-forcing number of fullerene graphs, MATCH Commun. Math. Comput. Chem. 74 (2015) 681--700.
\bibitem{Li} F. Zhang, X. Li, Hexagonal systems with forcing edges, Discrete Math. 140 (1995) 253--263.
\bibitem{ZhangD} H. Zhang, K. Deng, Forcing spectrum of a hexagonal system with a forcing edge, MATCH Commun. Math. Comput. Chem. 73 (2015) 457--471.
\bibitem{Ye} H. Zhang, D. Ye, W.C. Shiu, Forcing matching numbers of fullerene graphs, Discrete Appl. Math. 158 (2010) 573--582.
\bibitem{Zhang0} H. Zhang,  F. Zhang, Perfect matchings of polyomino graphs,
Graphs Combin. 13 (1997) 295--304.
\bibitem{Zhang} H. Zhang,  F. Zhang, Plane elementary bipartite graphs, Discrete Appl. Math. 105 (2000) 291--311.
\bibitem{ZhangZhao}  H. Zhang, S. Zhao, R. Lin, The forcing polynomial of catacondensed hexagonal systems, MATCH
Commun. Math. Comput. Chem. 73 (2015) 473--490.
\bibitem{Zhou2} H. Zhang, X. Zhou, A maximum resonant set of polyomino graphs, Discuss. Math. Graph Theory 36 (2016) 323--337.
\bibitem{ZhangQQ} Q. Zhang, H. Bian, E. Vumar, On the anti-kekul\'{e} and  anti-forcing number of cata-condensed phenylenes,
MATCH Commun. Math. Comput. Chem.  65 (2011) 799--806.
\bibitem{Zhaos} S. Zhao, H. Zhang, Forcing polynomials of benzenoid parallelogram and its related benzenoids, Appl. Math. Comput. 284 (2016) 209--218.
\bibitem{Zhaos1} S. Zhao, H. Zhang, Anti-forcing polynomials for benzenoid systems with forcing edges, Discrete Appl. Math. 250 (2018) 342--356.
\bibitem{Zhaos2} S. Zhao, H. Zhang,  Forcing and anti-forcing polynomials of perfect matchings for some rectangle grids, J. Math. Chem. 57 (2019) 202--225.
\bibitem{Zhou} X. Zhou, H. Zhang, Clar sets and maximum forcing numbers of hexagonal systems, MATCH Commun. Math. Comput. Chem. 74 (2015) 161--174.
\bibitem{Zhou1} X. Zhou, H. Zhang, A minimax result for perfect matchings of a polyomino graph, Discrete Appl. Math. 206 (2016) 165--171.
\end{thebibliography}
\end{document}